\documentclass[12pt]{article}
\usepackage{fullpage}
\usepackage{amssymb,amsfonts,amsmath,amsthm}
\usepackage{ stmaryrd}
	\theoremstyle{definition}
	\newtheorem{thm}{Theorem}
	\newtheorem{prop}[thm]{Proposition} 
	\newtheorem{lem}[thm]{Lemma}

\usepackage[parfill]{parskip}
\usepackage{enumerate}

\newcommand{\ignore}[1]{}

\begin{document}

\title{Noncommutative recursions and the Laurent phenomenon}
\author{Matthew C. Russell\footnote{Department of Mathematics, Rutgers, The State University of New Jersey, Piscataway, NJ
08854. Email: {\tt russell2 [at] math [dot] rutgers [dot] edu}}
}
\date{\today}
\maketitle

\begin{abstract}
We exhibit a family of sequences of noncommutative variables, recursively defined using monic palindromic polynomials in $\mathbb Q \left[x\right]$, and show that each possesses the Laurent phenomenon. This generalizes a conjecture by Kontsevich.
\end{abstract}

\section{Introduction}
Let $K_r$ (the Kontsevich map) be the automorphism of a noncommutative plane defined by 
\begin{equation*}
K_r:\left(x,y\right) \mapsto \left(xyx^{-1},\left(1+y^r\right)x^{-1}\right).
\end{equation*}
Maxim Kontsevich conjectured that, for any $r_1,r_2 \in \mathbb N$, the iterates 
\begin{align*}
\dots K_{r_2} K_{r_1} K_{r_2} K_{r_1} \left(x,y\right)
\end{align*}
are all given by noncommutative Laurent polynomials in $x$ and $y$. This is known as the Laurent phenomenon. The conjecture was proved in special cases for certain values of $r_1$ and $r_2$ (see~\cite{Us09},~\cite{Us10},~\cite{DiFK10}, and~\cite{DiFK11}), sometimes also with the positivity conjecture (that all of the Laurent polynomials have nonnegative integer coefficients), and sometimes replacing $1+y^r$ with any monic palindromic polynomial. Eventually, Berenstein and Retakh~\cite{BR11} gave an elementary proof of the Kontsevich conjecture for general $r_1$ and $r_2$. while Rupel~\cite{Rup12} subsequently proved it using the Lee-Schiffler Dyck path model (see~\cite{LS}) while also settling the positivity conjecture.

Later, Berenstein and Retakh~\cite{BR13} extended their methods to consider a more general class of recurrences given by $Y_{k+1} Y_{k-1} = h_{k}\left(a_{k-1,k}Y_k a_{k,k+1}\right)$, where $h_k \in \mathbb Q\left[x\right]$ and $h_k\left(x\right)=h_{k-2}\left(x\right)$ for all $k \in \mathbb Z$, $Y_1 a_{12} Y_2 a_{23} = a_{32} Y_2 a_{21} Y_1$, and $a_{k,k\pm 1}$ are defined recursively by $a_{k+2,k+1} = a_{k-1,k}^{-1}$ and $a_{k+1,k+2}=a_{k,k+1}^{-1}$. Proceeding in a similar fashion to their previous paper, they prove the Laurent phenomenon for these recurrences where $h_{k}=1+x^{r_k}$.

In this paper, we endeavor to expand the methods of Berenstein and Retakh~\cite{BR13} to higher-order recurrences, and to using monic palindromic polynomials instead of $1+x^r$.

\section{Preliminaries}
Following section 4 of Berenstein and Retakh~\cite{BR13}, let $K \ge 2$. Let $\mathcal F_K$ denote the $\mathbb Q$-algebra generated by $a_{1} ^{\pm 1}$, $a_{2} ^{\pm 1}$, $\dots$, $a_{K} ^{\pm 1}$, $b_{1} ^{\pm 1}$, $b_{2} ^{\pm 1}$, $\dots$, $b_{K} ^{\pm 1}$, and let $\mathcal F_K\left(Y_1,\dots,Y_K\right)$ denote the algebra generated by $\mathcal F_K$ and $Y_1,\dots,Y_K$, subject to the relation
\begin{equation}\label{eq:rel1}
Y_1 a_{1} Y_2 a_{2} \cdots Y_K a_{K} = b_{K} Y_K b_{K-1} Y_{K-1} \cdots b_{1} Y_1.
\end{equation}
For $n \in \mathbb{Z}$, define $a_{n}$ and $b_{n}$ recursively by
\begin{align} 
a_{n+K} &= b_{n} ^{-1} \label{eq:rec2} \\
b_{n+K} &= a_{n} ^{-1}. \label{eq:rec3}
\end{align}
Suppose we have a sequence of monic palindromic polynomials $h_n \in \mathbb Q\left[x\right]$ such that $h_n = h_{n-K}$ for all $n \in \mathbb{Z}$. Let us write $h_n\left(x\right) = \sum_{i=0} ^{d_n} P_{n,i} x^i$, so $P_{n,0}=P_{n,d_n}=1$ for all $n$. Define
\begin{align*}
&h_n ^{\downarrow} \left(x\right) \quad = \quad  \sum_{i=1} ^{d_n} P_{n,i} x^i \quad = \quad h_n \left(x\right)-1 \\
&h_n ^{\downarrow\downarrow} \left(x\right) \quad =\quad  \sum_{i=1} ^{d_n} P_{n,i} x^{i-1} \quad = \quad \frac{h_n \left(x\right)-1}x
\end{align*}
and recursively define $Y_n \in \mathcal F_K\left(Y_1,\dots,Y_K\right)$ for $n \in \mathbb{Z} \setminus \left\{1,\dots,K\right\}$ by 
\begin{equation} \label{eq:rec1}
Y_{n+K}Y_n = h_n\left(a_{n} Y_{n+1} a_{n+1} Y_{n+2} \cdots Y_{n+K-1} a_{n+K-1}\right).
\end{equation}
Define
\begin{align*}
Y_{n,m}^- &= a_{n-1} Y_n a_{n}Y_{n+1} \cdots a_{m-1} Y_m a_{m} & n\le m \\
Y_{n,m}^+ &= b_{n} Y_n b_{n-1}Y_{n-1} \cdots b_{m} Y_m b_{m-1} & n\ge m,
\end{align*}
while also defining $Y_{n,n-1}^- = a_{n-1}$ and $Y_{n,n+1}^+ = b_{n}$.
Then,~\eqref{eq:rec1} becomes
\begin{equation}\label{eq:rec6}
Y_{n+K} Y_n = h_n\left(Y_{n+1,n+K-1}^-\right).
\end{equation}
\begin{prop}
For all $n \in \mathbb{Z}$, we also have
\begin{align}
Y_n Y_{n+K} &= h_n\left(Y_{n+K-1,n+1} ^+ \right) \label{eq:rec4} \\
Y_n Y_{n+1,n+K-1}^- &= Y_{n+K-1,n+1}^+ Y_n \label{eq:rec5}.
\end{align}
\end{prop}
\begin{proof}
We proceed by induction on $n$. (We will only prove these for $n \ge 1$; the proof for $n<1$ is similar.) Note that~\eqref{eq:rel1} gives us the base case of~\eqref{eq:rec5} for $n=1$. Now, suppose that~\eqref{eq:rec5} holds for some $n\ge 1$.  Conjugating~\eqref{eq:rec6} on the left by $Y_n$ gives
\begin{align*}
Y_nY_{n+K} = Y_n h_n\left(Y_{n+1,n+K-1}^-\right) Y_n ^{-1},
\end{align*}
but, since
\begin{align*}
Y_n Y_{n+1,n+K-1}^- Y_n ^{-1} = Y_{n+K-1,n+1}^+ Y_n Y_n ^{-1}  = Y_{n+K-1,n+1}^+ 
\end{align*}
(by the inductive hypothesis), we conclude that
\begin{align*}
Y_n Y_{n+K} &= h_n\left(Y_{n+K-1,n+1} ^+ \right),
\end{align*}
which is~\eqref{eq:rec4}.

Now, we desire to prove~\eqref{eq:rec5} for $n+1$. 
This is equivalent to proving $Y_{n+1,n+K-1}^-Y_{n+K}=Y_{n+K}Y_{n+K-1,n+1}^+$
(multiply each side by $a_n^{-1}=b_{n+K}$ on the left and $a_{n+K} = b_n^{-1}$ on the right to recover the original expression).
We calculate
\begin{align*}
Y_{n+1,n+K-1}^-Y_{n+K} &= Y_n ^{-1} Y_{n+K-1,n+1} ^+  Y_n Y_{n+K} \\
&= Y_n ^{-1} Y_{n+K-1,n+1} ^+ h_n\left(Y_{n+K-1,n+1} ^+ \right) \\
&= Y_n ^{-1} h_n\left(Y_{n+K-1,n+1} ^+ \right) Y_{n+K-1,n+1} ^+  \\
&= Y_n ^{-1} Y_n Y_{n+K} Y_{n+K-1,n+1} ^+  \\
&=  Y_{n+K} Y_{n+K-1,n+1} ^+,
\end{align*}
using ~\eqref{eq:rec5} (the inductive hypothesis) and~\eqref{eq:rec4}.
\end{proof}

\section{Results}
Let $\mathcal A_n$ be the subalgebra of $\mathcal F_K\left(Y_1,\dots,Y_K\right)$ generated by $\mathcal F_K$ and $Y_n, \dots, Y_{n+2K-1}$. We now state our main result.
\begin{thm}
For all $n \in \mathbb{Z}$, $\mathcal{A}_n = \mathcal{A}_0$.
\end{thm}
\begin{proof}
It is enough to show that $\mathcal{A}_{n+1}=\mathcal{A}_n$. Without loss of generality, we let $n=0$. So, we try to show that $Y_{2K} \in \mathcal{A}_0$. By~\eqref{eq:rec4} and the definition of $h_K\left(x\right)$, we find
\begin{align} \notag
Y_{2K} &= Y_K ^{-1} h_K \left(Y_{2K-1,K+1} ^+\right) \\
	&= Y_K ^{-1}\sum_{i=0} ^{d_K-1} P_{K,i} \left(Y_{2K-1,K+1} ^+\right)^i +Y_K ^{-1}\left(Y_{2K-1,K+1} ^+\right)^{d_K}, \label{eq:first}
\end{align}
as $h_K$ is a monic palindromic polynomial.

We would like to find an expression for $Y_K ^{-1}\left(Y_{2K-1,K+1} ^+\right)^{d_K}$. Using~\eqref{eq:rec1}, we calculate
\begin{align} \notag
Y_0 &=Y_K^{-1}\left(1+\sum_{m=1} ^{d_0} P_{0,m} \left(Y_{1,K-1} ^-\right)^m\right) \\ \notag
Y_K^{-1} &=Y_0-Y_K^{-1}\sum_{m=1} ^{d_0} P_{0,m} \left(Y_{1,K-1} ^-\right)^m \\ \notag
Y_K^{-1}\left(Y_{2K-1,K+1} ^+\right)^{d_K} 
	&=Y_0\left(Y_{2K-1,K+1} ^+\right)^{d_K}
	-Y_K^{-1}\sum_{m=1} ^{d_0} P_{0,m} \left(Y_{1,K-1} ^-\right)^m\left(Y_{2K-1,K+1} ^+\right)^{d_K} \\
	&=Y_0\left(Y_{2K-1,K+1} ^+\right)^{d_K}
	-Y_K^{-1}\sum_{m=1} ^{d_0} P_{0,m} \left(Y_{1,K-1} ^-\right)^m\left(Y_{2K-1,K+1} ^+\right)^{m}\left(Y_{2K-1,K+1} ^+\right)^{d_K-m}. \label{eq:long}
\end{align}
Now, we would like a formula for $\left(Y_{1,K-1} ^-\right)^m\left(Y_{2K-1,K+1} ^+\right)^{m}$.

\begin{lem} \label{lem:tech}
For $m \ge 0$, we have
\begin{align*}
\left(Y_{1,K-1} ^-\right)^m \left(Y_{2K-1,K+1}^+\right)^m
	&= 1+\sum_{s=0}^{m-1} \left(Y_{1,K-1} ^-\right)^s
		\left(\sum_{j=1} ^{K-1} Y_{1,j-1} ^- h_j ^{\downarrow} \left(Y_{K+j-1,j+1} ^+\right)Y_{K+j-1,K+1}^+\right)
		\left(Y_{2K-1,K+1}^+\right)^s.
\end{align*}
\end{lem}
\begin{proof}
We proceed by induction on $m$. The case $m=0$ is trivial. For the case $m=1$, we calculate $Y_{1,K-1} ^-Y_{2K-1,K+1}^+$.
\begin{prop}
For $0 \le l$,
\begin{align*}
Y_{1,l}^- Y_{K+l,K+1}^+ = 1+\sum_{j=1}^l Y_{1,j-1}^- h_j ^\downarrow\left(Y_{K+j-1,j+1}^+\right)Y_{K+j-1,K+1}^+
\end{align*}
\end{prop}
\begin{proof}
The base case $l=0$ simply reduces to $a_{0}b_{K}=1$, which checks. 
Now, assuming the inductive hypothesis,
\begin{align*}
Y_{1,l+1}^-Y_{K+l+1,K+1}^+
&= Y_{1,l}^- Y_{l+1} a_{l+1} b_{K+l+1} Y_{K+l+1} Y_{K+l,K+1}^+ \\
&= Y_{1,l}^- Y_{l+1} Y_{K+l+1} Y_{K+l,K+1}^+ \\
&= Y_{1,l}^- h_{l+1}\left(Y_{K+l,l+2}^+\right)Y_{K+l,K+1}^+ \\
&= Y_{1,l}^- h_{l+1}^\downarrow \left(Y_{K+l,l+2}^+\right)Y_{K+l,K+1}^+ + Y_{1,l}^-Y_{K+l,K+1}^+ \\
&= Y_{1,l}^- h_{l+1}^\downarrow \left(Y_{K+l,l+2}^+\right)Y_{K+l,K+1}^+ + 1+\sum_{j=1}^l Y_{1,j-1}^- h_j ^\downarrow\left(Y_{K+j-1,j+1}^+\right)Y_{K+j-1,K+1}^+\\
&= 1+\sum_{j=1}^{l+1} Y_{1,j-1}^- h_j ^\downarrow\left(Y_{K+j-1,j+1}^+\right)Y_{K+j-1,K+1}^+,
\end{align*}
and we have proved our proposition.
\end{proof}
By this proposition, we see
\begin{align*}
Y_{1,K-1}^- Y_{2K-1,K+1}^+ = 1+\sum_{j=1}^{K-1} Y_{1,j-1}^- h_j ^\downarrow\left(Y_{K+j-1,j+1}^+\right)Y_{K+j-1,K+1}^+,
\end{align*}
which takes care of our base case. Proceeding to the inductive step, we calculate
\begin{align*}
&\left(Y_{1,K-1} ^-\right)^{m+1} \left(Y_{2K-1,K+1}^+\right)^{m+1} \\
&= \left(Y_{1,K-1} ^-\right)^m Y_{1,K-1} ^-Y_{2K-1,K+1}^+ \left(Y_{2K-1,K+1}^+\right)^m \\
&= \left(Y_{1,K-1} ^-\right)^m\left(1+\sum_{j=1}^{K-1} Y_{1,j-1}^- h_j ^\downarrow\left(Y_{K+j-1,j+1}^+\right)Y_{K+j-1,K+1}^+\right)\left(Y_{2K-1,K+1}^+\right)^m \\
&= \left(Y_{1,K-1} ^-\right)^m\left(Y_{2K-1,K+1}^+\right)^m
+ \left(Y_{1,K-1} ^-\right)^m\left(\sum_{j=1}^{K-1} Y_{1,j-1}^- h_j ^\downarrow\left(Y_{K+j-1,j+1}^+\right)Y_{K+j-1,K+1}^+\right)\left(Y_{2K-1,K+1}^+\right)^m \\
&= 1+\sum_{s=0}^{m-1} \left(Y_{1,K-1} ^-\right)^s
		\left(\sum_{j=1} ^{K-1} Y_{1,j-1} ^- h_j ^{\downarrow} \left(Y_{K+j-1,j+1} ^+\right)Y_{K+j-1,K+1}^+\right)
		\left(Y_{2K-1,K+1}^+\right)^s \\
&+ \left(Y_{1,K-1} ^-\right)^m\left(\sum_{j=1}^{K-1} Y_{1,j-1}^- h_j ^\downarrow\left(Y_{K+j-1,j+1}^+\right)Y_{K+j-1,K+1}^+\right)\left(Y_{2K-1,K+1}^+\right)^m \\
&= 1+\sum_{s=0}^{m} \left(Y_{1,K-1} ^-\right)^s
		\left(\sum_{j=1} ^{K-1} Y_{1,j-1} ^- h_j ^{\downarrow} \left(Y_{K+j-1,j+1} ^+\right)Y_{K+j-1,K+1}^+\right)
		\left(Y_{2K-1,K+1}^+\right)^s,
\end{align*}
which completes our proof.
\end{proof}

\begin{lem} \label{lem:another}
For $0 \le l \le K$ and $ l+1 \le q \le K+1$, we have
\begin{align} \label{eq:lemeq}
Y_{1,l} ^- Y_{K+l,q}^+ = Y_K Y_{K-1,q} ^+ \prod_{t=1} ^l h_t \left(b_{q-1} ^{-1} Y_{q-1,t+1} ^+ Y_{K+t-1,q}^+\right)
\end{align}
\end{lem}
\begin{proof}
If $l=0$, then~\eqref{eq:lemeq} trivially holds, as the product is empty. Otherwise, note that
\begin{align*}
Y_{1,l} ^- Y_{K+l,q}^+ &= Y_{1,l-1} ^- Y_l a_{l} b_{K+l} Y_{K+l} Y_{K+l-1,q}^+ \\
	&= Y_{1,l-1} ^- Y_l Y_{K+l} Y_{K+l-1,q}^+ \\
	&= Y_{1,l-1} ^- h_l\left(Y_{K+l-1,l+1}^+\right) Y_{K+l-1,q}^+ \\
	&= Y_{1,l-1} ^- Y_{K+l-1,q}^+ h_l\left(b_{q-1}^{-1}Y_{q-1,l+1} ^+ Y_{K+l-1,q}^+\right).
\end{align*}
Repeating this, we find
\begin{align*}
Y_{1,l} ^- Y_{K+l,q}^+
	&= Y_{1,0}^- Y_{K,q}^+ \prod_{t=1} ^{l} h_{t} \left(b_{q-1} ^{-1} Y_{q-1,t+1} ^+ Y_{K+t-1,q}^+\right) \\
	&= a_0 b_K Y_K Y_{K-1,q}^+ \prod_{t=1} ^{l} h_{t} \left(b_{q-1} ^{-1} Y_{q-1,t+1} ^+ Y_{K+t-1,q}^+\right) \\
	&= Y_K Y_{K-1,q}^+ \prod_{t=1} ^{l} h_{t} \left(b_{q-1} ^{-1} Y_{q-1,t+1} ^+ Y_{K+t-1,q}^+\right).
\end{align*}
\end{proof}

\begin{lem} \label{lem:flip}
For $s\ge 0$, we have $\left(Y_{1,K-1} ^-\right)^s Y_K = Y_K\left(Y_{K-1,1} ^+\right)^s $.
\end{lem}
\begin{proof}
Using~\eqref{eq:rel1} and~\eqref{eq:rec2}, we note that
\begin{align*}
Y_{1,K-1} ^- Y_K &= a_{0} Y_1 a_{1} \cdots Y_{K-1} a_{K-1} Y_K \\
	&= a_{0} \left(Y_1 a_{1} \cdots Y_{K-1} a_{K-1} Y_K a_{K} \right) a_{K} ^{-1} \\
	&= b_{K}^{-1} \left(b_{K} Y_K b_{K-1} Y_{K-1} \cdots b_{1} Y_1 \right) b_{0} \\
	&= Y_K b_{K-1} Y_{K-1} \cdots b_{1} Y_1 b_{0} \\
	&= Y_K Y_{K-1,1} ^+.
\end{align*}
The general claim follows by induction.
\end{proof}

\begin{lem} \label{lem:final}
For $m \ge 0$,
\begin{align*}
\left(Y_{1,K-1} ^-\right)^m \left(Y_{2K-1,K+1}^+\right)^m 
 = 1+ Y_K \sum_{s=0} ^{m-1}\sum_{j=1} ^{K-1} \left(Y_{K-1,1} ^+\right)^s
		 Y_{K-1,j+1} ^+ A\left(j,s\right)
\end{align*}
where
\begin{equation}
A\left(j,s\right)=\left(\prod_{t=1} ^{j-1} h_t \left(b_{j} ^{-1} Y_{j,t+1} ^+ Y_{K+t-1,j+1}^+\right)\right)
h_j ^{\downarrow \downarrow} \left(Y_{K+j-1,j+1}^+\right)Y_{K+j-1,K+1}^+
\left(Y_{2K-1,K+1} ^+\right)^s.
\end{equation}
\end{lem}
\begin{proof}
From Lemma~\ref{lem:another},
\begin{align*}
Y_{1,j-1}^- h_j ^\downarrow\left(Y_{K+j-1,j+1}^+\right)&=Y_{1,j-1}^- Y_{K+j-1,j+1}^+ h_j ^{\downarrow \downarrow}\left(Y_{K+j-1,j+1}^+\right) \\
&=Y_K Y_{K-1,j+1}^+ \left(\prod_{t=1}^{j-1} h_t\left(b_{j}^{-1} Y_{j,t+1}^+Y_{K+t-1,j+1}^+\right)\right) h_j ^{\downarrow \downarrow}\left(Y_{K+j-1,j+1}^+\right),
\end{align*}
and so, by Lemma~\ref{lem:flip},
\begin{align*}
& \left(Y_{1,K-1} ^-\right)^s Y_{1,j-1}^- h_j ^\downarrow\left(Y_{K+j-1,j+1}^+\right) \\
&=\left(Y_{1,K-1} ^-\right)^s  Y_K Y_{K-1,j+1} \left(\prod_{t=1}^{j-1} h_t\left(b_{j}^{-1} Y_{j,t+1}^+Y_{K+t-1,j+1}^+\right)\right) h_j ^{\downarrow \downarrow}\left(Y_{K+j-1,j+1}^+\right) \\
&=Y_K \left(Y_{K-1,1} ^+\right)^s Y_{K-1,j+1}\left( \prod_{t=1}^{j-1} h_t\left(b_{j}^{-1} Y_{j,t+1}^+Y_{K+t-1,j+1}^+\right)\right) h_j ^{\downarrow \downarrow}\left(Y_{K+j-1,j+1}^+\right)
\end{align*}
Regrouping and substituting the above into the expression from Lemma~\ref{lem:tech} gives the desired result.
\end{proof}

We now have our final expression for $\left(Y_{1,K-1} ^-\right)^m\left(Y_{2K-1,K+1} ^+\right)^{m}$. So, by~\eqref{eq:long} and Lemma~\ref{lem:final}, we find 
\begin{align*}
Y_K^{-1}\left(Y_{2K-1,K+1} ^+\right)^{d_K}
	&=Y_0\left(Y_{2K-1,K+1} ^+\right)^{d_K}-Y_K^{-1}\sum_{m=1} ^{d_0} P_{0,m}\left(Y_{2K-1,K+1} ^+\right)^{d_K-m} \\
	&-Y_K^{-1}\sum_{m=1} ^{d_0} P_{0,m}\left(Y_K \sum_{s=0} ^{m-1}\sum_{j=1} ^{K-1}  \left(Y_{K-1,1} ^+\right)^s
		Y_{K-1,j+1} ^+ A\left(j,d_K+s-m\right) \right)\\
&=Y_0\left(Y_{2K-1,K+1} ^+\right)^{d_K}-Y_K^{-1}\sum_{m=1} ^{d_0} P_{0,m}\left(Y_{2K-1,K+1} ^+\right)^{d_K-m} \\
	&-\sum_{m=1} ^{d_0}\sum_{s=0} ^{m-1}\sum_{j=1} ^{K-1}  P_{0,m} \left(Y_{K-1,1} ^+\right)^s
		Y_{K-1,j+1} ^+ A\left(j,d_K+s-m\right).
\end{align*}
Now, by~\eqref{eq:first},
\begin{align*}
Y_{2K} &= Y_K ^{-1}\sum_{i=0} ^{d_K-1} P_{K,i} \left(Y_{2K-1,K+1} ^+\right)^i +Y_0\left(Y_{2K-1,K+1} ^+\right)^{d_K}-Y_K^{-1}\sum_{m=1} ^{d_0} P_{0,m}\left(Y_{2K-1,K+1} ^+\right)^{d_K-m} \\
	&-\sum_{m=1} ^{d_0}\sum_{s=0} ^{m-1}\sum_{j=1} ^{K-1}  P_{0,m} \left(Y_{K-1,1} ^+\right)^s
		Y_{K-1,j+1} ^+ A\left(j,d_K+s-m\right).
\end{align*}
Using the facts that  $h_0=h_K$ (so $d_0=d_K$ and $P_{0,i}=P_{K,i}$ for all $i$) and $h_0$ is palindromic, we find
\begin{align*}
\sum_{m=1} ^{d_0} P_{0,m} \left(Y_{2K-1,K+1}^+ \right)^{d_K-m} &=
\sum_{m=1} ^{d_0} P_{0,d_0-m} \left(Y_{2K-1,K+1}^+ \right)^{d_K-m} \\
&= \sum_{i=0} ^{d_0-1} P_{0,i} \left(Y_{2K-1,K+1}^+ \right)^{i}\\
&= \sum_{i=0} ^{d_K-1} P_{K,i} \left(Y_{2K-1,K+1}^+ \right)^{i}.
\end{align*}
We then get the desired cancellation of terms, and can write
\begin{align}\label{eq:Y2Kfinal}
Y_{2K}  &= Y_0\left(Y_{2K-1,K+1} ^+\right)^{d_K} -\sum_{m=1} ^{d_0}\sum_{s=0} ^{m-1}\sum_{j=1} ^{K-1} P_{0,m} \left(Y_{K-1,1} ^+\right)^s Y_{K-1,j+1} ^+ A\left(j,d_K+s-m\right).
\end{align}

We have finally shown that $Y_{2K} \in \mathcal{A}_0$, which proves our main theorem.
\end{proof}
As a corollary, we see that these recursions have the Laurent phenomenon: each $Y_n$ is a noncommutative Laurent polynomial in $Y_1,\dots,Y_K$.

\section{Maple programs}

The Maple package {\tt NonComChecker} was written to accompany this paper. It is freely available at {\tt math.rutgers.edu/\~{}russell2/papers/recursions13.html}.

The main function, {\tt VerifyPaper(H,x)}, inputs a $H$, a list of $K$ polynomials in a variable $x$, which are taken to be
$h_1\left(x\right),h_2\left(x\right),\dots, h_K\left(x\right) = h_0\left(x\right)$. It then simplifies
$Y_{2K}$ using a list of equations that it generates, including~\eqref{eq:rel1},~\eqref{eq:rec2},~\eqref{eq:rec3},~\eqref{eq:rec6}, and~\eqref{eq:rec4},
 and verifies that the result equals~\eqref{eq:Y2Kfinal}.

\section{Acknowledgments}
The author would like to thank Vladimir Retakh and Doron Zeilberger for illuminating discussions.

\end{document}